\documentclass[reqno,oneside,dvipsnames,pdfnames,final]{amsart}

\makeatletter
\def\author@andify{%
  \nxandlist {\unskip ,\penalty-1 \space\ignorespaces}%
    {\unskip {} \@@and~}%
    {\unskip \penalty-2 \space \@@and~}%
}
\renewcommand{\andify}{%
  \nxandlist{\unskip, }{\unskip{} \@@and~}{\unskip \space \@@and~}}
\makeatother

\usepackage[letterpaper]{geometry}
\usepackage[final]{hyperref}
\usepackage[backend=biber,backref=true,hyperref=auto,doi=true,isbn=false,giveninits=true,url=false]{biblatex}
\usepackage[anythingbreaks]{breakurl} 

\usepackage{fourier}
\usepackage[final]{graphicx}
\usepackage{tikz}
\usetikzlibrary{cd,knots,arrows.meta}

\usepackage{amssymb,amsfonts,amsmath}
\usepackage{mathtools}

\usepackage{enumitem}
\usepackage{xcolor}
\usepackage{ifdraft}
\usepackage{todonotes}
\PassOptionsToPackage{cmyk}{xcolor}

\newtheorem{thm}{Theorem}[section]
\newtheorem{thrm}{Theorem}[section]

\newtheorem{cor}[thm]{Corollary}
\newtheorem{prop}[thm]{Proposition}
\newtheorem*{prop*}{Proposition}
\newtheorem{lem}[thm]{Lemma}

\newtheorem*{thm*}{Theorem}
\newtheorem*{lem*}{Lemma}

\theoremstyle{definition}

\newtheorem{exmp}[thm]{Example}

\newtheorem{notn}[thm]{Notation}

\newtheorem*{defn*}{defn}
\newtheorem*{fact*}{Fact}

\newtheorem{remk}[thm]{Remark}

\newtheorem*{rem*}{Remark}
\newtheorem*{rems*}{Remarks}

\newcommand{\Q}{\mathbb{Q}}
\newcommand{\Z}{\mathbb{Z}}

\newcommand{\C}{\mathbb{C}}

\DeclareMathOperator{\Isom}{Isom}

\DeclareMathOperator{\Spec}{Spec}
\DeclareMathOperator{\Br}{Br}

\DeclareMathOperator{\Tr}{Tr}

\newcommand{\GL}{\mathrm{GL}}
\newcommand{\SL}{\mathrm{SL}}
\newcommand{\PSL}{\mathrm{PSL}}
\newcommand{\PGL}{\mathrm{PGL}}
\newcommand{\et}{\mathrm{\acute et}}

\newcommand{\sh}[1]{\mathcal{#1}}
\newcommand{\tensor}{\otimes}
\newcommand{\Hoh}{\mathrm{H}}

\newcommand{\op}{\mathrm{op}}

\newcommand{\iso}{\cong}
\newcommand{\Mat}{\operatorname{Mat}}

\newcommand{\Nrd}{\operatorname{Nrd}}
\newcommand{\Trd}{\operatorname{Trd}}

\newcommand{\sm}{\setminus}

\newcommand{\id}{\mathrm{id}}

\newcommand{\End}{\operatorname{End}}

\newcommand{\PO}{\mathrm{PO}}

\newcommand{\fppf}{\ifmmode\mathrm{fppf}\else fppf\ \fi}

\usetikzlibrary{hobby,knots,decorations.markings}

\begin{document}

\title{Extending degree-$2$ Azumaya algebras with $C_2$-actions and examples from character varieties of knot groups}
\author{Justin Lawrence}
\address{Department of Mathematics\\
  The University of British Columbia\\
  BC V6T 1Z2\\
  CANADA}
\author{Nicholas Rouse}
\author{Ben Williams}
\email[J.~Lawrence]{justin@math.ubc.ca}
\email[B.~Williams]{tbjw@math.ubc.ca}

\begin{abstract}
  We give criteria to determine when a degree-$2$ Azumaya algebra with $C_2$-action over a dense open subvariety of a curve extends to the entire curve as an algebra with $C_2$-action. These consist of conditions for the extension of the algebra, combined with a new condition for the extension of the algebra with the action. The new condition is testable by computer algebra systems, and we explain how the result applies to the canonical components of the character varieties of certain hyperbolic knots with order-$2$ symmetries. We conclude by carrying out the calculations for different symmetries of the Figure-$8$ knot.
\end{abstract}
\thanks{We acknowledge the support of the Natural Sciences and Engineering Research Council of Canada (NSERC), RGPIN-2021-02603.}

\maketitle

Let $W \subseteq Y$ denote a scheme and a dense open subscheme. If $\sh A$ is an Azumaya algbera, i.e., an \'etale sheaf of matrix algebras, defined over $W$, one may ask whether $\sh A$ extends to an Azumaya algebra over $Y$. This question arises in the study of the $\SL(2)$-character variety of a hyperbolic knot group $\Gamma$, as detailed in \cite{Chinburg2020}. In this application, $Y$ is a canonical component of the character variety, and is therefore a $1$-dimensional curve defined over a number field. The subvariety $W$ is the dense open subvariety where the representation is irreducible, and the algebra $\sh A$ over $W$ is tautologically defined. There is a deep, and mysterious, conjecture \cite[Conjecture 6.9]{Chinburg2020} that the algebra $\sh A$ extends if and only if the knot in question is not an $L$-space knot.

In general, canonical components of character varieties of hyperbolic knot groups are complicated entities, being curves of high degree. In \cite{Boyle2024}, the authors exploit the symmetry of symmetric (in this case, freely periodic) hyperbolic knots to make deductions about the canonical components of the character varieties of the knot groups. We hope that, in general, character varieties of symmetric knots may be more tractable by virtue of the symmetry of the knot.

This paper considers the following situation: suppose as before that $W \subseteq Y$ denotes a scheme and a dense open subscheme, and that $\sh A$ is an Azumaya algebra of degree $2$ over $W$. Suppose further that $\sh A$ carries a nontrivial action by $C_2$, the cyclic group of order $2$, i.e.,  there is a nontrivial order-$2$ self map $g: \sh A \to \sh A$ of algebras over $W$. One may ask whether $\sh A$ extends as an algebra with $C_2$-action over all of $Y$. An obvious necessary condition is that $\sh A$ should extend over $Y$ disregarding the action, but this is not sufficient. In the case where $Y$ is a regular integral curve over a field of characterstic different from $2$, we can describe all further obstructions to the extension: they are certain square-classes in the residue fields of codimension-$1$ points of $Y \sm W$. The precise statement is in Theorem \ref{th:main} below.

We then give example calculations where $Y$ is a canonical component of a hyperbolic knot group carrying a $C_2$-action and $W$ is the subvariety of characters of absolutely irreducible representations. The algebra $\sh A$ over $W$ in this case is the tautological algebra, i.e., the sheaf whose stalk over a geometric point $w$ in $W$ is the target of a representation whose character is $w$. This algebra is constructed in a different way in \cite[Prop.\ 4.1]{Chinburg2020}, and we devote Section \ref{sec:examples} to constructing this algebra  $\sh A$ in such a way that it is clear it carries a $C_2$ action, and to producing a symbol-algebra presentation of the pullback of $\sh A$ to the fraction field $\Spec F \to W$. In Section \ref{sec:knots}, we use the geometry of the Figure-8 knot to give an example of such an algebra with $C_2$-actions over $W$ that both do and also do not extend over all of $Y$.

\subsection*{Acknowledgements}
\label{sec:acknowledgments}

We are grateful to Mike Bennett for helping us run Magma code. We are grateful to Keegan Boyle for conversations about and help in understanding symmetric knots.

\section{The structure group}
\label{sec:structure-group}

Let $W$ be a scheme. A \emph{degree-$2$ Azumaya algebra} (see \cite[\S 5]{Grothendieck1968}) on $W$ is a sheaf of algebras $\mathcal{A}$ that is locally isomorphic in the étale topology to $\Mat_{2\times 2}(\sh O_W)$: that is, for every $w \in W$, there exists an open $U \ni w$ and a finite étale map $U' \to U \subseteq W$ for which the pullback of $\sh A$ to $U'$ is isomorphic to $\Mat_{2 \times 2}(\sh O_{U'})$. In this circumstance, one says $U'$ is an \emph{étale neighbourhood} of $w$.

A \emph{$C_2$-action} on a degree-$2$ Azumaya algebra $\sh A$ is a map  $g : \sh A \to \sh A$ of sheaves of algebras on $W$ for which $g^2=\id_{\sh A}$. We emphasize that $g$ preserves the order of multiplication, and is therefore not an involution of Azumaya algebras.

\begin{remk}
  In our $C_2$-action, the underlying scheme $W$ is fixed. One can give a more general definition, in which $W$ itself has a $C_2$-action and $\sh A$ has a compatible action. This is much more complicated, and not considered further here.
\end{remk}

An Azumaya algebra over $W$ is necessarily a coherent sheaf. If $R$ is a ring, then an Azumaya algebra $\sh A$ on $\Spec R$ determines and is determined by the $R$-algebra of global sections. In a standard abuse of notation, we write $\sh A$ for this $R$-algebra.

In \cite[Thm.\ 5.1]{Grothendieck1968}, a number of different characterizations of Azumaya algebras over schemes are given. In the case of an affine scheme $\Spec R$, one such characterization amounts to the following: $\sh A$ is an $R$-algebra that is finitely generated and projective as an $R$-module, and for which the canonical map $\sh A \tensor_R \sh A^{\op} \to \End_R(\sh A)$ is an isomorphism. This is the definition used by \cite{Auslander1960}.

In the case where $R=F$ is a field, an Azumaya algebra over $R$ is a central simple algebra over $F$ and \textit{vice versa}. In particular, a degree-$2$ Azumaya algebra $A$ over a field $F$ is either a $4$-dimensional central division algebra over $F$ (the nonsplit case), or is isomorphic to $\Mat_{2 \times 2}(F)$ (the split case). In either case, we will call $A$ a \emph{quaternion algebra} over $F$.

\begin{exmp}\label{ex:standardAction}
  We define $J =
  \begin{bmatrix}
    0 & 1 \\ 1 & 0
  \end{bmatrix}$. Over any scheme $W$, there is a trivial degree-$2$ Azumaya algebra $\Mat_{2 \times 2}(\sh O_W)$, which can be endowed with the $C_2$-action $\iota$ given by conjugating by $J$.
\end{exmp}
Our first observation is that all nontrivial $C_2$-actions on degree-$2$ Azumaya algebras are étale locally isomorphic to that of Example \ref{ex:standardAction}.

\begin{prop} \label{pr:onlyOneOrder2ActionSHL}
  Let $W$ be a connected scheme on which $2$ is invertible. Let $\sh A$ be a degree-$2$ Azumaya algebra over $W$ and $g : \sh A \to \sh A$ a nontrivial order-$2$ isomorphism of $\sh A$. Locally in the étale topology, the algebra with $C_2$-action $(\sh A, g)$ is isomorphic to $(\Mat_{2 \times 2}(\sh O_W), \iota)$.
\end{prop}
\begin{proof}
  First we make a preliminary observation about nontriviality of the action. Since $2$ is invertible, we may decompose $\sh A$ as a locally free sheaf $\sh A_+ \oplus \sh A_{-}$, where $g$ acts as $\pm 1$ on $\sh A_{\pm}$. Since $W$ is connected , the ranks of the locally free sheaves $\sh A_+$ and $\sh A_-$ are constant. In particular, this implies that $g$ acts nontrivially on every stalk of $\sh A$.

  Let $w \in W$ be a point. We will prove that there is an étale neighbourhood $U'$ of $w$ such that the pullback of $(\sh A, g)$ to $U'$ is isomorphic as an algebra with automorphism to $\Mat_{2 \times 2}(\sh O_{U'}, \iota)$. We may suppose already that we have passed to an étale neighbourhood of $w$ over which $\sh A$ is isomorphic to a matrix algebra. That is to say, we may safely suppose $\sh A = \Mat_{2\times 2}(\sh O_W)$.

  Let $A$ denote the pullback of $\sh A$ to the (spectrum of the) residue field $\kappa(w)$ of $w$. The Skolem--Noether theorem tells us that the automorphism induced by $g$ on $A$ is given by conjugation by some element of $\GL(2; \kappa(w))$. This matrix is actually defined in some subring of $\kappa(w)$ that defines an affine open neighbourhood of $w$. That is, by passing to some neighbourhood $U_0$ of $w$, we may suppose that the automorphism $g$ is given by conjugation by some invertible matrix $C \in GL(2;\sh O_{U_0})$. Since $g^2 = \id$, it must be the case that $C^2= d I_2$ where $d\in \sh O^\times_{U_0}$. Working étale-locally we can assume that $d$ has a square-root on some étale neighbourhood $U_1$ of $w$, and so replace $C$ by a matrix satisfying $C^2 = I_2$, but where $C \neq I_2$ in each stalk. In a neighbourhood of $w$, the matrix $C$ may be brought to rational canonical form, which must be $C = J$ as required.
\end{proof}

Proposition \ref{pr:onlyOneOrder2ActionSHL} implies that isomorphism classes degree-$2$ Azumaya algebras with $C_2$-action on $W$ are given by the pointed set $\Hoh^1_\et(W; Q)$ where $Q$ is the automorphism group of $(\Mat_{2 \times 2}, \iota)$, a scheme of algebras-with-$C_2$-action over $\Z[1/2]$. The automorphism group of the algebra-scheme $\Mat_{2 \times 2}$ is the linear algebraic group $\PGL(2)$ acting by conjugation. The group $Q$ is the closed subgroup $\PGL(2)$ consisting of automorphisms that commute with $\iota$. If $k$ is a field of characteristic different from $2$, then the $k$-valued points of $Q$ are $k^\times$-equivalence classes of matrices 
$\begin{bmatrix}
  a & b \\ c & d
\end{bmatrix}$ over $k$
satisfying
\[
  \begin{bmatrix}
    a & b \\ c & d 
  \end{bmatrix}
  \begin{bmatrix}
    0 & 1 \\ 1 & 0 
  \end{bmatrix} = f \begin{bmatrix}
    0 & 1 \\ 1 & 0 
  \end{bmatrix} \begin{bmatrix}
    a & b \\ c & d 
  \end{bmatrix}\]
for some scalar $f \in k^\times$. Multiplying out, we see that $f = \pm 1$ and $a=fd$ and $b=fc$. The linear algebraic group $Q$ is a disconnected group of two components, and over an algebraically closed field is isomorphic to $\PO(2)$.

Being disconnected, $Q$ is not reductive according to the definition used in \cite{Colliot-Thelene1979}. Nonetheless, it is reductive in the sense of \cite{Mumford1994}, and so $Q \times X$ over $X$ meets the condition of \cite[Rem.\ 6.15]{Colliot-Thelene1979}, which in this case requires that there exist some principal $Q$-bundle $\GL(N) \to B$ for some integer $N$. Certainly $Q$ may be embedded as a closed subgroup in $\GL(4)$, and it is reductive since we assume the characteristic of $k$ is not $2$. The fact that $\GL(4) \to \GL(4)/Q$ is a principal $Q$-bundle now follows from \cite[Prop.\ 0.9 \& Amplification 1.3]{Mumford1994}.

\section{The local calculation}
\label{sec:local-calculation}

If $W$ is a scheme, $W^{(1)}$ denotes the set of codimension-$1$ points of $W$, i.e., those points $w \in W$ for which $\dim \sh O_{W,w}=1$. The following result is a variant of \cite[Cor.\ 6.14]{Colliot-Thelene1979}.

\begin{prop}[{\cite{Colliot-Thelene1979}}]\label{pr:CTS}
  Let $Y$ be a regular integral scheme of dimension $2$ or less, and let $W$ be a dense open affine subscheme. Write $F$ for the function field of $Y$. Let $\alpha$ be a class in $\Hoh^1_\et(W; Q)$. Then $\alpha$ is in the image of the restriction map $\Hoh^1_\et(Y; Q) \to \Hoh^1_\et(W; Q)$ if and only if its restriction to $\Hoh^1_\et(F; Q)$ along $\Spec F \to W$ is in the image of the map $\Hoh^1_\et(\Spec \sh O_{Y,y}; Q)  \to \Hoh^1_\et(\Spec F; Q)$ for all $y \in Y^{(1)}\sm W^{(1)}$.
\end{prop}
\begin{proof}
  The only-if direction is easy: suppose there exists a class $\alpha_Y \in \Hoh^1_\et(Y; Q)$ that restricts to $\alpha$. Then by functoriality, in the square
  \[
    \begin{tikzcd}
      \Hoh^1_\et(Y;Q) \rar \dar & \Hoh^1_\et(W;Q) \dar \\
      \Hoh^1_\et(\Spec \sh O_{Y,y}; Q) \rar & \Hoh^1_\et(\Spec F; Q) 
    \end{tikzcd},
  \] its restriction to $F$ is in the image of the map $\Hoh^1_\et(\sh O_{Y,y}; Q) \to \Hoh^1_\et(F; Q)$.
  
In the other direction, the result is implicit in \cite[\S6]{Colliot-Thelene1979}. First of all, since $Y$ is a regular integral scheme of dimension $2$ or less, \cite[Cor.~6.13]{Colliot-Thelene1979} applies, which says that \cite[Prop.~6.8]{Colliot-Thelene1979} holds for this $Y$. That proposition says something stronger than what we assert.

The hypothesis on $\alpha$ implies that its restriction $\alpha_F$ to $F$ can serve as a class $\xi$ in the proof of \cite[Prop.~6.8]{Colliot-Thelene1979}. Furthermore, our $W$ can serve as $V_0$ in the proof of that proposition. That proof then proceeds to enlarge the subset of $Y$ on which $\xi$ is defined iteratively, until it is defined on $Y$ itself, so that a class $\alpha_Y \in \Hoh^1_\et(Y;Q)$ is produced that maps to $\alpha \in \Hoh^1(W;Q)$.
\end{proof}

Suppose $Y$ is a regular integral curve, $W$ is a dense open subvariety, and $(\sh A,g)$ is an Azumaya algebra with $C_2$-action defined on $W$. By using \ref{pr:CTS}, we reduce the problem of extending $(\sh A, g)$ to all of $Y$ to a (finite) set of local problems: define $(A,g)$ to be the restriction of $(\sh A, g)$ to the fraction field of $W$ and $Y$. For each point $y \in Y\sm W$, determine whether one can extend $(A,g)$ to the local ring $\sh O_{Y,y}$, which is a discrete valuation ring in $F$.

\begin{notn}
  If $R$ is a domain with field of fractions $F$, and $A$ is a quaternion algebra over $F$---which we allow to be split, i.e., isomorphic to $\Mat_{2 \times 2}(F)$---then an \emph{$R$-order} in $A$ is an $R$-subalgebra $O$ of $A$ with the property that $F \tensor_R O = A$. See \cite[Ch.~10]{Voight2021} for instance.
\end{notn}

Let $R$ be a domain with field of fractions $F$. Suppose $(A,g)$ is a quaternion algebra with $C_2$-action over a field $F$. The question of whether the associated cohomology class $[A,g] \in \Hoh^1(\Spec F; Q)$ is in the image of $\Hoh^1(\Spec R; Q) \to \Hoh^1(\Spec F; Q)$ is precisely the question of whether $(A,g)$ admits a $g$-invariant order $O$ that is Azumaya over $R$. This may in general be a difficult question to answer, but when $R$ is a discrete valuation ring, we can answer it.

The following lemma is an amalgamation of results of Auslander and Goldman. We state it here for ease of reference.
\begin{lem} \label{lem:AusGold}
  Suppose $R$ is a discrete valuation ring (DVR) and $F$ is its field of fractions. If $A$ is a central simple algebra of degree $n$ over $F$ such that the Brauer class of $A$ is in the image of $\Br(R) \to \Br(F)$, then every maximal order in $A$ is Azumaya over $R$.
\end{lem}
\begin{proof}
  Choose a maximal $R$-order $O$ for $A$. It is the case that $O$ is a projective $R$-module by \cite[Prop.~1.3 and the Corollary to Prop.~4.7]{Auslander1960}. 
  
  It remains to show that $O$ is central and separable over $R$. The Brauer group condition on $[A]$ assures us that there exists some algebra $B$ that is central and separable over $R$ and for which $B \tensor_{R} F$ is isomorphic to $A \tensor_F \Mat_{n \times n}(F)$ for some integer $n$. Choose an identification $B \tensor_{R} F = A \tensor_K\Mat_{n \times n}(F)$. This makes $B$ into a maximal order in $A \tensor_F \Mat_{n \times n}(F)$, and $\Mat_{n \times n}(O)$ is also a maximal order in this $F$-algebra (see  \cite[Thm 3.6]{Auslander1960}), so that \cite[Prop.~3.5]{Auslander1960} tells us these orders are conjugate. Then \cite[Thm 3.5]{Auslander1960a} implies that $O$ is central and separable over $R$, completing the proof.
\end{proof}

\begin{prop} \label{pr:DVR} Let $R$ be a DVR, with fraction field $F$ of characteristic different from $2$ and residue field $\kappa$. Let $(A,g)$ be a quaternion algebra with $C_2$-action over $F$. Necessarily the action of $g$ on $A$ is given by conjugation by some element $a \in A$. Then $(A,g)$ admits a $C_2$-invariant Azumaya order if and only if both the following hold:
  \begin{enumerate}
  \item \label{i:1}  The Brauer class of $A$ is in the kernel of the residue map $\Br(F) \to \Hoh^1_\et(\kappa; \Q/\Z)$;
  \item \label{i:2} The valuation of the reduced norm $\Nrd(a)$ is even.
  \end{enumerate}  
\end{prop}
\begin{proof}
  Condition \ref{i:1} holds if and only if the Brauer class $[A]$ is in the image of the map $\Br(R) \to \Br(F)$. Therefore, if $A$ admits an Azumaya maximal order $O$, it is necessary and sufficient that the Brauer class of $A$ be in the kernel of $\Br(F) \to \Hoh^1_\et(\kappa; \Q/\Z)$.

  If $O$ is $C_2$-invariant, then $a^{-1}Oa = O$ gives an automorphism of $O$, which is an Azumaya algebra over a local ring. All automorphisms of such algebras are inner, so that there exists some $b \in O^\times$ such that conjugation by $a$ and by $b$ agree. Therefore $a=bf$ for some $f \in F$, and taking reduced norms gives $\Nrd(a) = f^2\Nrd(b)$. Since $b \in O^\times$, the valuation $v(\Nrd(b))$ is $0$ by \cite[Lemma 10.3.7]{Voight2021}, and so $\Nrd(a) \equiv 0 \pmod 2$.

   Conversely, suppose the class of $[A]$ is in the image of $\Br(R) \to \Br(F)$ and that $v(\Nrd(a)) \equiv 0 \pmod 2$, where $v$ is the valuation. By hypothesis, we can find some unit $u \in F^\times$ so that $\Nrd(ua) = u^2 \Nrd(a)$ lies in $R^\times$.  Since conjugation by $a$ and $ua$ give the same automorphism of $A$, and since $v(\Nrd(a)) \equiv v(\Nrd(ua)) \pmod 2$, we may suppose without loss of generality that $\Nrd(a) \in R^\times$, i.e., $v(\Nrd(a)) = 0$. We wish to construct an Azumaya maximal order $O$ in $A$ for which $a \in O^\times$.

  From $g^2 = \id$, we deduce that $a^2 \in F^\times1_A$. Since $a$ satisfies its minimal polynomial
  \[ x^2 - \Trd(a) x + \Nrd(a) 1_A = 0 \]
  but $a$ is not itself central in $A$, and therefore does not satisfy any linear equation with coefficients in $F$, we deduce that $a^2 = -\Nrd(a)1_A$ and $\Trd(a) = 0$. One says that the element $a \in A$ is \emph{integral} over $R$ (see, e.g., \cite[Def.~10.3.1]{Voight2021}).

  If $A$ is split, i.e., if we can identify $A = \Mat_{2\times 2}(F)$, then $O$ can be found directly. There exists some $s \in \GL(2;F)$ for which $s^{-1}as$ is in rational canonical form:
  \[ s^{-1}as =
    \begin{bmatrix}
      0 & 1 \\ \Nrd(c) & 0
    \end{bmatrix}. \]
  Then the $R$-order
  \[ O =  s \Mat_{2 \times 2}(R) s^{-1} \]
  Is seen immediately to contain $a$.

  Suppose $A$ is not split, which is to say, it is a quaternion division algebra. We claim that there is an $R$-order $B \subset A$ containing $a$. We may write $A$ as a symbol algebra, i.e., $A=F \oplus Fi \oplus Fj \oplus Fij$, where $i^2, j^2 \in F$ and $ij=-ji$. At least one of $i$, $j$ lies outside $F \oplus Fa$. Without loss of generality, $i$ does. It is not possible for both of $ia$ and $ai$ to lie in $F \oplus Fa \oplus Fi$, since $A$ has no $3$-dimensional $F$-subalgebras. Without loss of generality, suppose $F \oplus Fa \oplus Fi \oplus Fai = A$. Let $B$ be the \emph{left order} of the $R$-lattice $L =R \oplus Ra \oplus Ri \oplus Rai$, i.e., the set of $x \in A$ for which $x L \subseteq L$. This is an order of $A$, \cite[p.~109]{Reiner1975}, and contains $a$. If $ia$ is used in place of $ai$, then the right order $\{ x \in A \mid Lx \subseteq L \}$ should be used instead.
  
  Once an order $B \ni a$ is found, we may produce a maximal order $O$ containing $a$, which is Azumaya by Lemma \ref{lem:AusGold}. Since $\Nrd(a)$ is a unit in $R$, the element $a^{-1} = \frac{1}{\Nrd(a)} a$ lies in $O$ as well. Therefore, the Azumaya $R$-algebra $O$ is $C_2$-invariant, as required.
\end{proof}

\begin{thrm}\label{th:main}
  Suppose $Y$ is a regular integral curve over $k$, and $W$ is a dense open subvariety. Let $(\mathcal{A}, g)$ be an Azumaya algebra of degree $2$ with $C_2$-action over $W$. Write $A$ for the quaternion algebra given by restricting $\sh A$ to $\Spec F$, the fraction field of $Y$. Choose some $a \in A$ so that the automorphism induced by $g$ on $A$ is given by conjugation by $a$. Then $(\mathcal{A},g)$ extends to some Azumaya algebra of degree $2$ with $C_2$ action over $Y$ if and only if the following both hold for all points $y \in Y \sm W$:
  \begin{itemize}
  \item the Brauer class $[A] \in \Br(F)$ is in the kernel of the residue map $\Br(F) \to \Hoh^1_\et(\kappa(y); \Q/\Z)$
  \item the valuation of $\Nrd(a)$ with respect to the valuation on $\kappa(y)$ is even.
  \end{itemize}
\end{thrm}
\begin{proof}
  Proposition \ref{pr:CTS} tells us that $(\sh A, g)$ extends to $Y$ if and only if $(A,g)$ over $\Spec F$ extends to each of $\Spec \sh O_{Y,y}$ as $y$ ranges over the points of $Y \sm W$. Since $\sh O_{Y,y}$ is a DVR, determining the extension to $\Spec \sh O_{Y,y}$ is handled by Proposition \ref{pr:DVR}. This gives the result.
\end{proof}

\section{Character Varieties of Dimension 1}
\label{sec:examples}

For the rest of the paper, we assume the ground field is $\Q$. Our examples all arise in the following way. Let $\Gamma$ be a finitely generated discrete group, and let $h \colon \Gamma \to \Gamma$ be an order-$2$ automorphism of $\Gamma$. Define $R(\Gamma)$ to be the $\SL(2;\Q)$-representation variety\footnote{We use the term ``variety'' in the sense of a reduced separated scheme of finite type over the ground field.} of $\Gamma$: a reduced finite type $\Q$-scheme whose closed points are homomorphisms $\phi: \Gamma \to \SL(2;k)$ for some field extension $k/\Q$. We will replace $R(\Gamma)$ by an irreducible $C_2$-invariant component, which we denote $R(\Gamma)_0$.

The scheme $R(\Gamma)_0$ carries two actions: an action by the group-scheme $\PGL(2)$ over $\Q$ given by change-of-basis in $\SL(2)$, and a $C_2$-action given by precomposition by $h$. These actions commute.

We may take a GIT quotient
\[ R(\Gamma)_0 \to R(\Gamma)_0/\PGL(2) =: X(\Gamma)_0  \]
to form $X(\Gamma)_0$, which is an irreducible component of the character variety, $X(\Gamma)$, of $\Gamma$. The action of $C_2$ on $R(\Gamma)_0$ descends to give an action of $C_2$ on $X(\Gamma)_0$. We will continue this introduction under the assumption that the $C_2$-action on $X(\Gamma)_0$ is trivial, since this is what happens in our examples.

The action of $\PGL(2)$ on $R(\Gamma)_0$ is particularly well behaved on the open $C_2$-invariant subscheme $V$ consisting of absolutely irreducible representations. Here the action of $\PGL(2)$ is free and proper, so that the quotient $q: V \to V/\PGL(2) := W$ is a principal $\PGL(2)$-bundle. We refer to \cite[Prop.\ 1.13]{Porti2017} for the proof. The map $q$ is $C_2$-equivariant, and $W$ is an open subvariety of $X(\Gamma)_0$ (on which we assume $C_2$ acts trivially).

\begin{notn}
  Suppose $\phi: \Gamma \to \SL(2;\bar \Q)$ is a representation. For all $\gamma  \in \Gamma$, there is a regular function
  \[ \gamma \mapsto \Tr(\phi(\gamma))  \]
  on $R(\Gamma)$. This function is conjugation-invariant and therefore descends to give a regular function $I_\gamma$ on  the character variety $X(\Gamma)$. It is a fact that the $I_\gamma$ serve to generate the coordinate ring of $X(\Gamma)$ (see, e.g., \cite[Prop.\ 2]{Sikora2013}), but we do not rely on this.
\end{notn}

The following calculation is useful in detecting (absolutely) reducible representations.
\begin{lem}\label{lem:reducibility}
  Suppose $\phi : \Gamma \to \SL(2; \bar \Q)$ is a reducible representation and
  $g, h$ are conjugate elements in $\Gamma$. Then either 
  \[ I_{gh}(\phi)=2 \quad \text{ or } \quad I_{gh}(\phi) = I_{g}(\phi)^2 -2.\]
\end{lem}
\begin{proof}
  Since $\phi$ is reducible, there exists a basis of $\bar \Q^2$ in which $\phi(g)$ and $\phi(h)$ are both in upper-triangular
  form. Since $g$, $h$ are similar, we have
  \[ \phi(g) =
    \begin{bmatrix}
      x & \ast \\ 0 & x^{-1} 
    \end{bmatrix}, \quad \phi(h) =
    \begin{bmatrix}
      x^{\pm} & \ast \\ 0  & x^{\mp} 
    \end{bmatrix}.
  \]
  The result is now a simple calculation.
\end{proof}

We will routinely write things such as ``the locus $I_{gh}= I_g^2-2$'' to mean the closed subvariety where this condition holds.

Since $q:V \to W$ is a principal $\PGL(2)$-bundle, and $\PGL(2)$ is the automorphism group of the  algebra-scheme $\Mat_{2 \times 2}$ over $\Q$, we can form an Azumaya algebra $\sh A$ (the \emph{tautological algebra}) on $W$ as the sheaf of sections of the fibre bundle
\[
  \begin{tikzcd}
    \Mat_{2\times 2} \rar & (\Mat_{2 \times 2} \times  V)/\PGL(2) \dar{q} \\ & W,
  \end{tikzcd}
\]
the quotient being taken by the diagonal $\PGL(2)$ action $A \cdot (M, \phi) = (A^{-1}MA, A^{-1}\phi A)$.
The $C_2$-action on $V$ induces a $C_2$-action on $\sh A$ over $W$, giving us an Azumaya algebra with $C_2$-action over $W$.

\begin{notn}\label{not:taut}
 If $\gamma \in \Gamma$ is an element, then there is a morphism $\tilde m_\gamma: V \to \Mat_{2 \times 2} \times V$ given by sending a representation $\phi$ to $(\phi(\gamma), \phi)$. This is $\PGL(2)$-equivariant: $A^{-1}\phi A \mapsto (A^{-1} \phi(\gamma) A, A^{-1} \phi A)$. Therefore it descends to give a morphism on the quotient
  \[ m_\gamma : W \to  (\Mat_{2 \times 2} \times  V)/\PGL(2), \]
  i.e., $m_\gamma  \in \sh A(W)$.   
\end{notn}

\subsection{Interpreting the tautological algebra}
\label{sec:interpr-taut-algebra}

We attempt to explain the sense in which $\sh A$ is tautological. It may be best to begin by considering the pullback of $\sh A$ along $q: V \to W$. This sheaf $q^*(\sh A)$ on $V$ is the sheaf of sections of $\Mat_{2 \times 2} \times V \to V$, i.e., it is a constant sheaf of $2 \times 2$ matrix algebras. Since a $\bar \Q$-valued point $\tilde w: \Spec \bar \Q \to V$ is just a representation $\Gamma \to \SL_{2 \times 2}(\bar \Q)$, and $\SL_{2 \times 2}(\bar \Q)$ is a subset of $\Mat_{2 \times 2}(\bar \Q)$, we may view $q^*(\sh A)$ as the sheaf of matrix algebras in which the representations of $\Gamma$ take values.

A point $w : \Spec \bar \Q  \to W$ is a conjugacy class $[\phi]$ of representations $\Gamma \to \SL(2;\bar \Q)$. The pullback $w^*(\sh A)$ is a sheaf of algebras over $\Spec \bar \Q$, i.e., an algebra. For any $w$, there exists a lift $\tilde w$ making the following diagram commute:
\begin{equation*}
  \begin{tikzcd}
    & V \dar{q} \\ \Spec \bar \Q \urar{\tilde w} \rar{w} & W
  \end{tikzcd}
\end{equation*}
The choice of $\tilde w$ is exactly the choice of a representative $\phi \in [\phi]$.

We may identify $\tilde w^*(q^*(\sh A)) = w^*(\sh A)$. We deduce that $w^*(\sh A)$ is abstractly isomorphic to $\Mat_{2 \times 2}(\bar \Q)$ and obtain an isomorphism $w^*(\sh A) \iso \Mat_{2 \times 2}(\bar \Q)$. We may view $w^*(\sh A)$ as the target matrix algebra of a representation in the class of $[\phi]$, but without a particular choice of representation, or equivalently, without a particular choice of basis for the algebra.

That is, the sheaf $\sh A$ is the sheaf of algebras over $W$ in which the representations, whose equivalence classes are the points in $W$, take values. There being no chosen bases for these algebras, we can access elements of them only by using the sections $m_\gamma$ for $\gamma \in \Gamma$, and polynomials in these sections. Typically, one cannot find global formulas in the $m_\gamma$ that everywhere produce a basis for the stalks $w^*(\sh A)$, so that $\sh A$ is not globally isomorphic to the split algebra $\Mat_{2 \times 2}(\sh O_W)$; rather it is a sheaf of algebras that is globally twisted, i.e., it is an nonsplit Azumaya algebra.

 Unwinding definitions, we see that $q^*(m_\gamma) = \tilde m_\gamma$. Furthermore $\tilde w^*(\tilde m_\gamma) = \phi(\gamma) \in \Mat_{2\times 2}(\bar \Q)$. Applying the relation $w^* = \tilde w^* \circ q^*$ again, we see that \[ w^*(m_\gamma) \in w^*(\sh A) \]
is the coordinate-independent version of
\[ \phi(\gamma) \in \Mat_{2\times 2}(\bar \Q), \]
in that the former yields the latter once a specific $\phi$ in the conjugacy class of representations is chosen.

Since the representations $\phi$ under consideration are all irreducible, their images in $\Mat_{2\times 2}(\bar \Q)$ generate it as an algebra. This being a coordinate-independent fact, the algebras $q^*(\sh A)$ are all generated by the elements $q^*(m_\gamma)$.

\subsection{The Algebra over the Fraction Field}
\label{sec:ff}

Our main theorem, Theorem \ref{th:main}, requires us to know something about the restriction of an Azumaya algebra $\sh A$ to the fraction field of the variety over which it is defined. In this section, we study this algebra.

We continue to assume that $W$ is a dense open subvariety of an irreducible component of a character variety $X(\Gamma)$ and that the tautological algebra $\sh A$ is defined over $W$. Let $F$ denote the fraction field of $W$, and write $A$ for the restriction of $\sh A$ to $\Spec F$. The algebra $A$ is an Azumaya algebra of degree $2$ over the field $F$, i.e., a quaternion algebra. In principle $A$ may be split, in the sense that $A \iso \Mat_{2\times 2}(F)$, but $A$ is  nonsplit in all the examples considered in this paper.

A quaternion algebra $A$ over a field $F$ of characteristic different from $2$ admits a presentation as a \emph{symbol algebra}, i.e., a fixed isomorphism with an algebra determined by a Hilbert symbol:
\[ A \iso \left( \frac{a,b}{F} \right ),\quad a,b \in F^\times. \]
This means that there are elements $i, j \in A$ satisfying $i^2=a$, $j^2=b$ and $ij=-ji$. We refer to \cite{Voight2021} for general background on such algebras. A number of features of $A$, notably the reduced norm and reduced trace, are easily computed when $A$ is presented as a symbol algebra. We shall give formulas for presenting the algebras that arise in this paper as symbol algebras.

Since $A$ is a degree-$2$ Azumaya algebra over a field, it carries a standard
involution: $\bar x = \Trd(x)1_A - x$. In these terms, $\Nrd(x)1_A = x \bar x$.
We may define an $F$-bilinear form on $A$:
\[ \langle x, y \rangle = \frac{1}{2}\Trd(\bar x y ) = \frac{1}{2}(\bar x y + \bar y x),\] which is immediately seen to be symmetric.

We will refer below to \emph{traceless} elements of $A$. These are elements for which $\bar x = -x$, so that $\Trd(x) = 0$. For such an element, $x^2 + \Nrd(x)1_A = 0$, so that $x^2 \in F$. Two traceless elements are orthogonal if and only if they anticommute, since we can write
\[ xy + yx = -\bar x y - \bar y x = 2 \langle x, y \rangle \]
under the assumption that the traces are $0$.

If $\gamma \in \Gamma$ is a group element, then $m_\gamma \in \sh A(W)$ was defined in Notation \ref{not:taut}. The global section $m_\gamma$ yields an element of $A$ by restriction, and we also use $m_\gamma$ for this element in an abuse of notation. The coordinate ring of $W$, which contains the element $I_\gamma$, is a subring of $F$, so that $I_\gamma$ is also an element of $F$. By virtue of the definitions, $I_\gamma = \Trd(m_\gamma)$.

\begin{prop} \label{pr:symbolAlgebra}
 Suppose that $m_g$ and $m_h$ generate $A$ as an algebra, and that $I_g^2
  \neq 4$. Define
  \[ i = 2m_g  - I_g, \quad j' = 2m_h - I_h, \quad j= -i\left(j' - \frac{\langle i,j' \rangle}{\langle i , i \rangle }
      i \right).\]
  Then $i,j$ generate $A$ as an $F$-algebra, $ij=-ji$ and
  \[ i^2 = I_g^2 - 4, \quad j^2 = -(I_g^2 -4)(I_h^2-4) + (2I_{gh} - I_gI_h)^2. \]
  In particular, $A$ is a symbol algebra given by
  \[ A \iso \left( \frac{I_g^2 - 4, \, -(I_g^2 -4)(I_h^2-4) + (2I_{gh} - I_gI_h)^2}{F} \right).\]  
\end{prop}

\begin{proof}
Starting with the generators $m_g, m_h$ of $A$, we may take the traceless elements
\[ i = 2m_g - I_g,\qquad j'= 2 m_h -  I_h, \]
and then orthogonalize with respect to the bilinear form:
\[ j'' = j' - \frac{ \langle i, j' \rangle}{\langle i, i \rangle} i. \]

\noindent We replace $j''$ by
\[ j = -i j''. \]
Together, $i, j$ are generators of the $F$-algebra $A$. Since $i$ and $j''$ are orthogonal and $j''$ is traceless,  $i$ and $-j'' = \bar \jmath''$ are also orthogonal. We calculate
\[ \Trd(j) = -ij'' + (-\bar \jmath'' \bar \imath) = -ij'' -j''i =  -2 \langle i, j'' \rangle = 0. \]
Furthermore,
\[ ij = -i^2 j'' = ij''i = -ji. \]
so that $i,j$ anticommute, i.e., are themselves orthogonal.

Therefore, $A \iso \left( \frac{ i^2, j^2 }{F} \right)$. It remains to calculate $i^2$ and $j^2$. One of these is easily written down directly:
 \[ i^2 = 4 m_g^2 - 4I_g m_g + I_g^2 = I_g^2 - 4. \]

Calculating $j^2$ is considerably more time-consuming. To simplify the notation, let us write
\[ Q = \frac{ \langle i, j' \rangle}{\langle i, i \rangle}. \] We may calculate explicitly

\begin{equation}
  \begin{aligned}  Q & = \frac{\Trd(i \bar \jmath')}{\Trd(i\bar \imath)} = \frac{\Trd(ij')}{\Trd(i^2)}\\
    &= \frac{\Trd(4m_gm_h -2m_gI_h -2m_hI_g +I_gI_h)}{\Trd(4m_g^2 -4m_gI_g +I^2_g) } \\
    &= \frac{2I_{gh} - I_g I_h}{I_g^2-4}.
  \end{aligned}
  \label{eq:17}
\end{equation}

\noindent Using this, we may write
\begin{align*} j^2 &= -(I_g^2-4)(b'')^2= (I_g^2-4)[(I_h^2 -4) - Q(ij' + j'i) + Q^2 (I_g^2-4)] \\
                   &= -(I_g^2-4)\left[(I_h^2 -4) -
                     \frac{2I_{gh}-I_gI_h}{I_g^2-4}(\Trd(4m_{gh}-2I_gm_h-2I_hm_g+I_gI_h))+ \right. \\
                     & \qquad \left.+\frac{(2I_{gh}-I_gI_h)^2}{(I_g^2-4)^{2}}(I_g^2-4)\right] \\
                   &= -(I_g^2-4)(I_h^2-4) +2(2I_{gh}-I_gI_h)^2-(2I_{gh}-I_gI_h)^2 \\
                   &=-(I_g^2-4)(I_h^2-4)  + (2I_{gh}-I_gI_h)^2.
\end{align*}
This completes the proof.
\end{proof}

We may also wish to write $m_g$ and $m_h$ in terms of $i,j$:
\begin{equation}
  \label{eq:18}
  m_g = \frac{1}{2}I_g + \frac{1}{2} i
\end{equation}
is immediate. Rearranging $j = -i j''$, we have 
\[ j''= \frac{-i}{I_g^2-4}j  \]
from which we deduce
\[ j'=\frac{-i}{I_g^2-4}j+Qi,  \]
so
\begin{equation}
  m_h= \frac{1}{2}I_h+ \frac{1}{2}Qi - \frac{1}{2I_g^2-8}ij\label{eq:19}
\end{equation}
where $Q$ is defined as in \eqref{eq:17} above.

\begin{cor} \label{cor:symbolConjugate}
  If $g$ and $h$ are conjugate, then the formulas of Proposition \ref{pr:symbolAlgebra} simplify to give
  \[i^2= I_g^2 -4, \quad j^2 = 4(I_{gh}-2)(I_{gh}-I_g^2 +2)  = 4I_{[g,h]}+16. \]
\end{cor}
\begin{proof}
  If $g$ and $h$ are conjugate, then $I_g=I_h$. From there, the simplification $j^2 = 4(I_{gh}-2)(I_{gh}-I_g^2 + 2)$ is elementary. Then using a trace relation, \cite[\S3.4, \,
  (3.15)]{MR}, we see that
  \[ 4 I_{[g,h]} = 8I^2_g + 4I^2_{gh} -4I_g^2 I_{gh} - 8 = j^2-16,\]
  which is what was required.
\end{proof}

\begin{remk}
  Since the isomorphism class of the symbol algebra $\left( \frac{x,y}{F} \right)$ is invariant under multiplication of $x,y$ by squares, our calculation agrees with \cite[Cor.~2.9]{Chinburg2020}, which says the quaternion algebra $A$ has Hilbert symbol
  \[ A \iso \left( \frac{I_g^2 -4 , I_{[g,h]} - 2}{F}\right). \] In this presentation, one takes the generators to  be what we have called $i$ and $j/2$.
\end{remk}

The condition in Proposition \ref{pr:symbolAlgebra} that $m_g$ and $m_h$ should generate $A$ is easily satisfied in practice.

\begin{prop}\label{pr:twogen}
  Suppose $W$ contains the character of a faithful representation. If $g, h \in \Gamma$ are two noncommuting elements, then $m_g, m_h$ generate the algebra $A$ over $F$.
\end{prop}
\begin{proof}
   Let $w \in W$ be the character of a faithful representation. If $g,h \in \Gamma$ are two noncommuting elements, then $m_gm_h \neq m_hm_g$ as global sections of $\sh A$, since their restrictions to $\sh A_w$ differ. Therefore $m_g m_h \neq m_hm_g$ in the quaternion algebra $A$ over $F$. A quaternion algebra is generated by any two noncommuting elements.
\end{proof}

\section{Examples from knot theory}
\label{sec:knots}

In our examples, $\Gamma$ is the knot group of the Figure-8 knot, which is a hyperbolic $2$-bridge knot $K \subset S^3$. In particular, the knot group $\Gamma = \pi_1(S^3 \sm K)$ admits a presentation with two conjugate generators and one relation.

The actions of $C_2$ on $\Gamma$ are induced from actions of $C_2$ on $S^3$ preserving $K$ setwise and preserving the orientation of $S^3$. For the group $\Gamma$, the character variety contains $1$ or $2$ distinguished $1$-dimensional components: the canonical components. In the case of the Figure-$8$ knot, there is $1$ canonical component, which contains  all characters of irreducible representations.

We take $X(\Gamma)_0$ to be such a canonical component, which is so named because it contains the class of a lift $\phi : \Gamma \to \SL(2;\C)$ of the canonical representation of $\Gamma$ as deck transformations of the universal cover $\widetilde{S^3 \sm K} \to S^3 \sm K$. The deck-transformation representation $\Gamma \to \Isom(\widetilde{S^3 \sm K})$ amounts to an irreducible faithful representation $\Gamma \to \PSL(2; \C) = \PGL(2;\C)$, once an isomorphism of $\widetilde{S^3 \sm K}$ with oriented hyperbolic $3$-space is chosen. The representation $\phi$ is actually defined over $\bar \Q$ and is an irreducible faithful representation: a clear account of this may be found in \cite[\S\ 4]{Dunfield2012}. We may therefore suppose that $W \subset X(\Gamma)_0$ contains the class $[\phi]$ of an absolutely irreducible faithful representation of $\Gamma$.

The variety $W$ is smooth over $\C$, but $X(\Gamma)_0$ may be singular. We replace $X(\Gamma)_0$ by $\tilde X$, its normalization, as in \cite{Chinburg2020}. The action of $C_2$ induces an action on the representation variety $R(\Gamma)$ and therefore on the irreducible component $R(\Gamma)_0$, and so on $X(\Gamma)_0, \tilde X$ and $W$. In the examples we consider, the action on $W$ is trivial, so that $C_2$ acts on the Azumaya algebra $\sh A$ over $W$.

Some of the calculations in this section were done using the computer algebra system \textit{Magma} \cite{Magma}. The code is available at \url{https://github.com/tbjw/Conjugating_Elements/releases/tag/arXiv20251026}. The results can be verified by hand.

\begin{remk}
  The question of whether one can tell \textit{a priori\/} that the action of $C_2$ on $\tilde X$ is trivial, knowing only the action of $C_2$ on $K \subseteq S^3$ is not addressed here. We are content to observe that it is trivial in our example by direct calculation.
\end{remk}

\subsection{The Figure-$8$ Knot}
\label{sec:figure-8-knot}

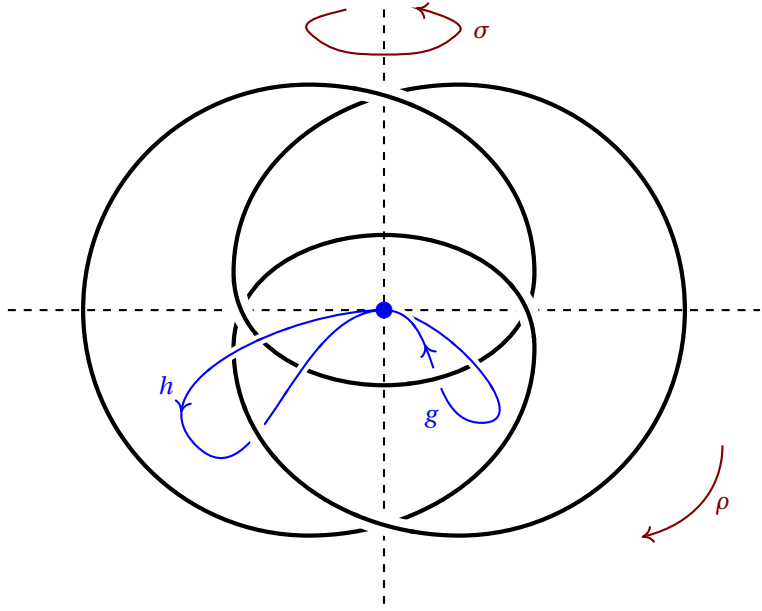
\begin{figure}
  \centering
  \begin{tikzpicture}[pics/arrow/.style={code={%
  \draw[line width=0pt,{Computer Modern Rightarrow[line
  width=0.8pt,width=1.5ex,length=1ex]}-] (-0.5ex,0) -- (0.5ex,0);
}}]

\draw [thick, dashed] (4,4)--(4,-4);

\draw [thick, dashed] (-1,0)--(9,0);

  \begin{knot}[
    consider self intersections=no splits,
    ignore endpoint intersections=true,
  clip width=5,
  clip radius=10pt,
  end tolerance=1pt 
  ]

\strand[ultra thick] (0,0)
to [out=up, in=left]  (3,3)
to [out=right, in=up]  (6,0.5)
to [out=down, in=right] (4,-1)
to [out=left, in=down] (2,0.5)
to [out=up, in=left] (5,3)
to [out=right, in=up] (8,0)
to [out=down, in=right] (5,-3)
to [out=left, in=down] (2,-0.5)
to [out=up, in=left] (4,1)
to [out=right, in=up] (6,-0.5)
to [out=down, in=right] (3,-3)
to [out=left, in=down] (0,0);

\strand[thick, blue] (4,0)
to [out=right, in=left] pic[pos=0.4,sloped]{arrow}(5.3,-1.5)
to [out=right, in=right] (4,0);

\strand[thick, blue] (4,0)
to [out=left, in=135]  pic[pos=0.8,sloped]{arrow}(1.5,-1.8)
to [out=315, in=left] (4,0);
\flipcrossings{2,4,6,7,9}

\strand[thick, red!50!black](3.5,4)
to [out=195, in=135](3,3.7)
to [out=315, in=left](4,3.4)
to [out=right, in=225](5,3.7)
to [out=45, in=165] pic[pos=1, sloped]{arrow}(4.5,4);

\strand[thick, red!50!black](7.5,-3)
to [out=15, in=270] pic[pos=0, sloped]{arrow}(8.5,-1.8);

\end{knot}
\node[blue] at (4.63,-1.45) {$g$};
\node[blue] at (1.1,-1.0) {$h$};
\node[red!50!black] at (5.3,3.7) {$\sigma$};
\node[red!50!black] at (8.5,-2.6) {$\rho$};
\filldraw[blue] (4,0) circle (3pt);
\end{tikzpicture}
\caption{The Figure-8 Knot, generators for the knot group, and two indicated symmetries.}
  \label{fig:Figure8}
\end{figure}

Let $K$ denote the Figure-8 Knot in $S^3$, depicted in Figure \ref{fig:Figure8}. The knot group admits a presentation
\[ \Gamma =\langle g, h\mid h^{-1} g^{-1} hgh^{-1}ghg^{-1}h^{-1}g \rangle. \]
Note that $g$ and $h$ are conjugate in this group, since $h = (g^{-1}hgh^{ -1}) g (g^{-1}hgh^{-1})^{-1}$.

The $\SL(2, \C)$-character variety is determined in \cite{Whittemore1973}. In fact, computer algebra calculations show that the character scheme over $\Q$ is given by the same equations. It consists of two components: one corresponding to abelian representations, which are not irreducible, and one canonical component $X_0$ which is a nonsingular affine curve defined by a single equation in $2$
variables $I_g,I_{gh}$ satisfying:
\begin{equation} \label{eq:whittemore}
  -I_g^2I_{gh}+2I_g^2+I_{gh}^2 - I_{gh} -1=0.
\end{equation}

The condition $I_{gh}=2$ is inconsistent with \eqref{eq:whittemore}. A representation $\phi$ is therefore absolutely irreducible unless $I_{gh}(\phi)^2 = I_g(\phi)^2 -2$, by Lemma \ref{lem:reducibility}. Combining this condition with equation \eqref{eq:whittemore}, we see that the absolutely irreducible locus contains the complement in $X_0$ of the closed variety given by $I_g^2 = 5$. This subvariety consists of two $\Q[\sqrt{5}]$-valued points, and we see that the corresponding representations are reducible by constructing them:
\[\phi(g) = \frac{1}{2} \begin{bmatrix}  \pm \sqrt{5} - 1 & 0 \\  0 & \pm \sqrt{5} +1 \end{bmatrix}, \quad \phi(h) = \frac{1}{2} \begin{bmatrix} \pm \sqrt{5} - 1 & 0 \\  -2  & \pm \sqrt{5} +1 \end{bmatrix}.
\]
Therefore, the absolutely irreducible locus $W\subseteq X_0$ in is precisely the open complement of the two points defined by $I_g^2 = 5$.

The tautological algebra $\sh A$ over $W$ is generated as an algebra by the global sections
$m_g, m_h$, by virtue of Proposition \ref{pr:twogen}. Letting $F$ denote the fraction field of $W$ and $A = \sh A_F$ the restriction of $\sh A$ along $\Spec F \to W$, Corollary \ref{cor:symbolConjugate} tells us that
\[ A = \left( \frac{ I_g^2 - 4, 4(I_{gh}-2)(I_{gh}-I_g^2+2)}{F} \right). \]
This algebra is generated by $i,j$, of course, and is also generated by $m_g$ and $m_h$, which may be expressed in terms of $i,j$ by means of equations \eqref{eq:18} and \eqref{eq:19}. Note that $I_g=I_h$ in this variety because $g,h$ are conjugate.

The Figure-$8$ knot admits several different $C_2$-actions. We will consider three of them: a strong inversion $\rho$ and two $2$-periodic symmetries $\sigma$, $\sigma \rho$. In Figure \ref{fig:Figure8}, the strong inversion $\rho$ is given by rotation by $\pi$ about an axis emanating directly out of the diagram at the marked central point. The $2$-periodicities $\sigma$ and $\rho \sigma$ are given by rotation by $\pi$ about the two dashed axes. Note that these two axes each intersect the knot at $2$ points.

\subsubsection*{The strong inversion}
\label{sec:strong-inversion}

The strong inversion acts on $\Gamma$ in the following way:
  \begin{equation}
    \label{eq:2}
    \rho : g \mapsto g^{-1}hg, \qquad  h \mapsto hg^{-1}hgh^{-1}.
  \end{equation}
This is determined by direct calculation from the diagrams in Figures \ref{fig:Figure8} and \ref{fig:Figure8-p}.

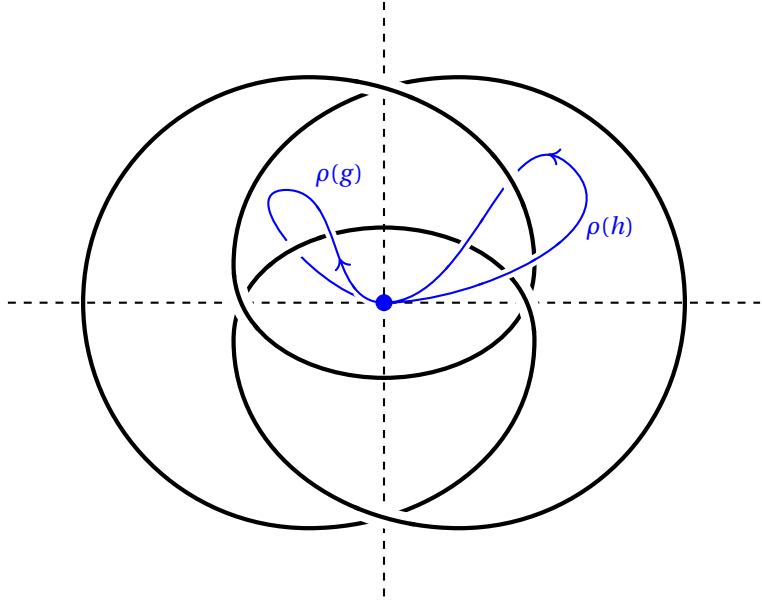
\begin{figure}[h]
  \centering
  \begin{tikzpicture}[pics/arrow/.style={code={%
  \draw[line width=0pt,{Computer Modern Rightarrow[line
  width=0.8pt,width=1.5ex,length=1ex]}-] (-0.5ex,0) -- (0.5ex,0);
}}]

\draw [thick, dashed] (4,4)--(4,-4);

\draw [thick, dashed] (-1,0)--(9,0);

\begin{knot}[
  consider self intersections=no splits,
  ignore endpoint intersections=true,
  clip width=5,
  clip radius=10pt,
  end tolerance=1pt 
  ]
  
\strand[ultra thick] (0,0)
to [out=up, in=left]  (3,3)
to [out=right, in=up]  (6,0.5)
to [out=down, in=right] (4,-1)
to [out=left, in=down] (2,0.5)
to [out=up, in=left] (5,3)
to [out=right, in=up] (8,0)
to [out=down, in=right] (5,-3)
to [out=left, in=down] (2,-0.5)
to [out=up, in=left] (4,1)
to [out=right, in=up] (6,-0.5)
to [out=down, in=right] (3,-3)
to [out=left, in=down] (0,0);

\strand[thick, blue] (4,0)
to [out=left, in=right] pic[pos=0.4,sloped]{arrow}(2.7,1.5)
to [out=left, in=left] (4,0);

\strand[thick, blue] (4,0)
to [out=right, in=135]  pic[pos=0.9,sloped]{arrow}(6.5,1.8)
to [out=315, in=right] (4,0);
\flipcrossings{3,4,7,9,10}

\end{knot}
\node[blue] at (3.4,1.7) {$\rho(g)$};
\node[blue] at (7.0,1.0) {$\rho(h)$};
\filldraw[blue] (4,0) circle (3pt);
\end{tikzpicture}
\caption{The Figure-8 Knot and generators for the knot group after application of $\rho$.}
\label{fig:Figure8-p}
\end{figure}

The induced action on $\sh A$ is given by
\[
   m_g \mapsto m_g^{-1} m_h m_g, \qquad m_h \mapsto m_h m_g^{-1} m_h m_g m_h^{-1}.
\]
We restrict to the fraction field $F$. Here there is an induced action $\rho_* : A \to A$ given by the same formulas. Since $\rho_*$ is an automorphism of a quaternion algebra, the Skolem--Noether theorem asserts it is given by conjugation by some $c \in A^\times$. That is, there exists some invertible (i.e., nonzero) $c$ for which
\[ c m_g  =  \rho_*(m_g) c , \qquad c m_h  =  \rho_*(m_h) c, \]
which amounts to a system of linear equations for the coefficients of $c$ over the field $F$. By writing everything in terms of $1,i,j, ij$ and using a computer algebra package, this system may be solved in order to find a nonzero element $c$. One such element is given by
\[ c = \frac{-4I_{gh}^2 + 16 I_{gh} - 16}{I_{gh}^2 - 3 I_{gh} + 3} i + \frac{-I_{gh}^2 + 5 I_{gh} -
    7}{I_{gh}^2 - 3 I_{gh} + 3} I_g j + k. \]
Once $c$ has been found, we apply Theorem \ref{th:main} by calculating the valuation of $\Nrd(c)$ associated with either of the points in $X_0 \sm W$, i.e., where $I_{gh}=I_g^2-2$. Again, this can be carried out using a computer algebra package. We see that the valuation of one such $c$, and therefore any such $c$, is even, so that Theorem \ref{th:main} tells us that the Azumaya algebra with $C_2$-action $\sh A$ extends to $X_0$ as an algebra with $C_2$-action.

\subsubsection{The $2$-periodicities}
\label{sec:2-period}

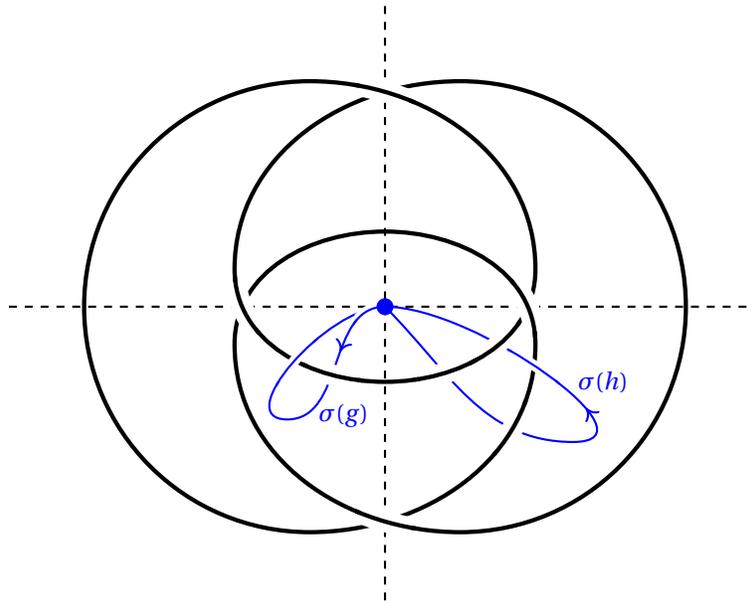
\begin{figure}[h]
  \centering
  \begin{tikzpicture}[pics/arrow/.style={code={%
  \draw[line width=0pt,{Computer Modern Rightarrow[line width=0.8pt,width=1.5ex,length=1ex]}-] (-0.5ex,0) -- (0.5ex,0);
}}]

\draw [thick, dashed] (4,4)--(4,-4);

\draw [thick, dashed] (-1,0)--(9,0);

  \begin{knot}[
    consider self intersections=no splits,
    ignore endpoint intersections=true,
  clip width=5,
  clip radius=10pt,
  end tolerance=1pt 
  ]

\strand[ultra thick] (0,0)
to [out=up, in=left]  (3,3)
to [out=right, in=up]  (6,0.5)
to [out=down, in=right] (4,-1)
to [out=left, in=down] (2,0.5)
to [out=up, in=left] (5,3)
to [out=right, in=up] (8,0)
to [out=down, in=right] (5,-3)
to [out=left, in=down] (2,-0.5)
to [out=up, in=left] (4,1)
to [out=right, in=up] (6,-0.5)
to [out=down, in=right] (3,-3)
to [out=left, in=down] (0,0);

\strand[thick, blue] (4,0)
to [out=left, in=right] pic[pos=0.4,sloped]{arrow}(2.7,-1.5)
to [out=left, in=left] (4,0);

\strand[thick, blue] (4,0)
to [out=right, in=right]  pic[pos=0.7,sloped]{arrow}(6.5,-1.8)
to [out=left, in=315] (4,0);
\flipcrossings{2,7,9}

\end{knot}
\node[blue] at (3.45,-1.45) {$\sigma(g)$};
\node[blue] at (6.9,-1.0) {$\sigma(h)$};
\filldraw[blue] (4,0) circle (3pt);
\end{tikzpicture}
\caption{The Figure-8 Knot and generators for the knot group after application of  $\sigma$.}
  \label{fig:Figure8-s}
\end{figure}

The 2-periodicity $\sigma$ acts on $\Gamma$ in the following way:
\begin{equation}
  \label{eq:sigma}
    \sigma : g \mapsto g^{-1}, \qquad  h \mapsto g^{-2}h^{-1}g^2.
  \end{equation}
This is determined by direct calculation from the diagrams in Figures \ref{fig:Figure8} and \ref{fig:Figure8-s}.

The automorphism of $A$ induced by $\sigma$ is given by conjugation by some element $c \in A^\times$ with the property that
\begin{equation}
  m_g c = cm_g^{-1}, \quad m_h c = cm_g^{-2}m_h^{-1}m_g^2.\label{eq:15sigma}
\end{equation}
As in the case of the strong inversion, equation \eqref{eq:15sigma} gives us $8$ linear equations for the $4$ coefficients of
\[ c = c_0 + c_1 i + c_2 j + c_3 ij \in A.\]
This may again be solved by use of a computer algebra package. One such value is
\[ c= \frac{-I_{gh}^2 + 3 I_{gh} - 3}{I_{gh}^2 - I_{gh} - 1} I_g j + k \] and the valuation of $\Nrd(c)$ at $X_0\sm W$ can be computed. In this case, we discover that the valuation is odd. From Theorem \ref{th:main}, we therefore deduce that the algebra $\sh A$ with $C_2$-action $\sigma_*$ does not extend over $X_0$ as an algebra with $C_2$-action, although it does extend without the action.

For the other $2$-periodicity, $\rho\sigma=\sigma\rho$, we may argue as follows: the conjugating element $c$ in this case may be taken to be the product of the conjugating elements for $\rho$ and $\sigma$. Since the reduced norm and the valuations are multiplicative, we see that the valuation of this conjugating element is odd at $X_0 \sm W$. Therefore, the algebra also fails to extend with the $C_2$-action $(\sigma \rho)_\ast$.

\printbibliography

\end{document}